\documentclass[a4paper,12pt]{amsart}
\usepackage{amssymb}
\usepackage{amsmath}
\usepackage{enumitem}
\usepackage{mathrsfs}
\usepackage[all]{xy}
\usepackage{marginnote}

\usepackage[margin=1.5in]{geometry}

\usepackage[OT2,T1]{fontenc}
\DeclareSymbolFont{cyrletters}{OT2}{wncyr}{m}{n}
\DeclareMathSymbol{\Sha}{\mathalpha}{cyrletters}{"58}

\RequirePackage{mathrsfs} 

\newtheorem{theorem}{Theorem}[section]
\newtheorem{lemma}[theorem]{Lemma}
\newtheorem{proposition}[theorem]{Proposition}
\newtheorem{corollary}[theorem]{Corollary}

\theoremstyle{definition}

\newtheorem{definition}[theorem]{Definition}

\numberwithin{equation}{section} \numberwithin{figure}{section}

\DeclareMathOperator{\Pic}{Pic} 
 
\DeclareMathOperator{\Aut}{Aut} \DeclareMathOperator{\Spec}{Spec}
 \DeclareMathOperator{\rank}{rank}

 \DeclareMathOperator{\Res}{R}

\DeclareMathOperator{\GL}{GL}
\newcommand{\PGL}{\textrm{PGL}}

\newcommand{\SO}{\textrm{SO}}
\newcommand{\SP}{\textrm{SP}}

\newcommand{\SL}{\textrm{SL}}

\newcommand{\Par}{\textrm{Par}}

\newcommand{\Qbar}{\overline{\QQ}}

\newcommand\FF{\mathbb{F}}
\newcommand\PP{\mathbb{P}}

\newcommand\ZZ{\mathbb{Z}}
\newcommand\NN{\mathbb{N}}
\newcommand\QQ{\mathbb{Q}}

\newcommand\OO{\mathcal{O}}

\renewcommand{\leq}{\leqslant}

\renewcommand{\geq}{\geqslant}

\title[Good reduction of algebraic groups and flag varieties]{Good reduction of algebraic groups and flag varieties}

\author{A. Javanpeykar}
\address{A. Javanpeykar \\
Institut f\"{u}r Mathematik\\
Johannes Gutenberg-Universit\"{a}t Mainz\\
Staudingerweg 9, 55099 Mainz\\
Germany.}
\email{peykar@uni-mainz.de}

\author{D. Loughran}
\address{D. Loughran \\
Leibniz Universit\"{a}t Hannover,
Institut f\"{u}r Algebra, Zahlentheorie
    und Diskrete Mathematik\\
Welfengarten 1\\
30167 Hannover\\
Germany.}
\email{loughran@math.uni-hannover.de}



\begin{document}

\begin{abstract}
	In 1983, Faltings proved that there are only finitely many abelian
	varieties over a number field of fixed dimension 
	and with good reduction outside a given set of places. 
	In this paper, we consider the analogous problem for other algebraic
	groups and their homogeneous spaces, such as flag varieties.	  
\end{abstract}

\maketitle

\thispagestyle{empty}

\section{Introduction}

An important classical result  in algebraic number theory is the theorem
of Hermite and Minkowski: there are only finitely many number fields
of bounded degree over $\QQ$ which are unramified outside  a given finite set $S$ of primes of $\QQ$.
It was first noticed by Shafarevich    \cite{Shaf1962} that such finiteness statements occur elsewhere in number
theory.  In particular, he  conjectured an analogous statement for curves over number fields. In his famous paper on Mordell's conjecture \cite{Fal83}, Faltings proved Shafarevich's conjecture, and also showed a corresponding finiteness statement for abelian varieties. Namely that, for $K$ a number field, $g$ an integer and $S$ a finite set of finite places of $K$, the set of $K$-isomorphism classes of $g$-dimensional abelian varieties over $K$ with good reduction outside $S$ is finite.

It is natural to ask whether the analogue of Faltings's result  holds for other
algebraic groups, and more generally for their homogeneous spaces (e.g. torsors).
This brings us to the first result of this paper.

\begin{theorem} \label{thm:reductive}
	Let $K$ be a number field, let $n\in \NN$ and let $B \subset \Spec \OO_K$ be a dense
	open subscheme. Then the set of $B$-isomorphism classes of pairs $(G,E)$,
	where $G$ is an $n$-dimensional reductive group scheme over $B$ and $E$
	is a $G$-torsor, is finite.
\end{theorem}

Here we say that $(G_1,E_1)$ is $B$-isomorphic to $(G_2,E_2)$ if $G_1 \cong G_2$ as
group schemes over $B$ and there exists a $B$-isomorphism $E_1 \cong E_2$
which respects the action of $G_1$ and $G_2$.

In this paper all reductive group schemes are connected (see Definition~\ref{def:reductive}).
The analogous statement of Theorem \ref{thm:reductive} \emph{fails} for disconnected groups;
for example the constant group scheme $\ZZ/2\ZZ$
over $\ZZ$ admits infinitely many non-isomorphic subgroup schemes, given by simply deleting
the non-identity component above each prime $p$. 
Taking $E=G$ in Theorem \ref{thm:reductive}, we obtain the following
immediate corollary.

\begin{corollary} \label{cor:reductive_sch}
	Let $K$ be a number field, let $n\in \NN$ and let $B \subset \Spec \OO_K$ be a dense 
	open subscheme.
	Then the set of $B$-isomorphism classes of $n$-dimensional reductive group schemes over $B$
	is finite.
\end{corollary}

Corollary \ref{cor:reductive_sch} can be used to obtain a result analogous to Faltings's theorem. 
Namely, let $S$ be a finite set of finite places of a number
field $K$. We say that a reductive algebraic group $G$
over $K$ has good reduction outside  $S$ if there exists a reductive group scheme over $\OO_K[S^{-1}]$ whose generic fibre is $K$-isomorphic to $G$. 

\begin{corollary}\label{cor:reductive_gr}
	Let $n\in \NN $ and let $S$ be a finite set of finite  places of a  number field $K$.
	Then the set of $K$-isomorphism classes of $n$-dimensional reductive algebraic
	groups over $K$ with good reduction outside  $S$
	is finite.
\end{corollary}

For motivation, let us briefly explain some special cases of our results.
Firstly, for each field extension $K \subset L$
of degree $n$ of a number field $K$, there is an $n$-dimensional algebraic
torus $\Res_{L/K} (\mathbb G_{m,L})$ over $K$ given by the Weil restriction of $\mathbb G_{m,L}$.
Here for such a torus to have good reduction outside $S$, the corresponding
field $L$ must be unramified outside  $S$. By Hermite-Minkowski
there are only finitely many such fields, in particular, there
are indeed only finitely many such tori. Secondly, for 
a central simple algebra $D$ of dimension $n^2$ over $K$, we may consider
the corresponding special linear group $\SL(D)$.
If $\SL(D)$ has good reduction outside $S$, then the central simple algebra $D$ must be unramified outside
 $S$. However by class field theory, there are only finitely
many such central simple algebras of fixed dimension, as expected.

These discussions should make clear that our proofs will 
require certain finiteness results for the cohomology
of group schemes. These will be provided to us by a 
theorem of Gille and Moret-Bailly \cite{GMB13} (which in turn
builds upon work of Borel and Serre \cite{BS64}). This result encapsulates
for example the finiteness of Tate-Shafarevich groups of linear algebraic
groups. 

The assumption that the group be \emph{reductive} in our results
is crucial; in \S \ref{sec:unipotent} we shall construct explicit counter-examples which show
that the analogues here for unipotent group schemes fail.

Reductive algebraic groups admit other natural homogeneous spaces (besides torsors),
namely flag varieties. As an application of our methods
we prove the following, which was in fact the original motivation for this paper.

\begin{theorem}\label{thm:flags}
	 Let $n\in \NN $ and let $S$ be a finite set of finite  places of a number field $K$.
	Then the set of $K$-isomorphism classes of flag varieties over $K$
	of dimension $n$  with good reduction outside $S$ is finite.
\end{theorem}

Here by \emph{good reduction}, we mean good reduction as a flag variety 
(see Definition \ref{def:good_reduction_flag}). Theorem \ref{thm:flags} generalises some already known results.
Namely, the analogue of Theorem \ref{thm:flags} was already known for Brauer-Severi varieties \cite[Thm.~5.3]{Scholl} and quadric hypersurfaces in projective space \cite{JL}. New special cases include, for example, finiteness results for twists of Grassmannians and twists of products of projective spaces.





The authors have recently initiated a program of study
into the ``Shafarevich conjecture for Fano varieties''. The case of Fano varieties of dimension $1$ (i.e.~conics)
is classical and the case of dimension $2$ is dealt with in \cite{Scholl}. The case of flag varieties,
handled in Theorem \ref{thm:flags}, is in some respects the next easiest case in this programme. In \cite{JL} the authors have proven other cases of this
for certain complete intersections in projective space. These papers should all be viewed as certain aspects of the same research programme.

Let $n \in \NN$. Most of our results generalise from algebraic groups of dimension $n$
over number fields to arbitrary global fields, provided
that one avoids certain small characteristics depending on $n$. For example,
take $K=\FF_p(x)$ for some prime $p$. Then
$K$ admits infinitely many Artin-Schreier extensions $L$ of degree $p$
ramified only at $\infty$ (see e.g.~\cite[\S 6.4]{Sti09}).
Moreover one can check that the $p$-dimensional tori $\Res_{L/K} (\mathbb G_{m,L})$ give rise to infinitely many
$K$-isomorphism classes with good reduction outside of $\infty$ 
(see \cite[Ex.~7.2.8]{Conrad} for the case of semi-simple groups).

There is no common generalisation of our results and the results
of Faltings, applying to groups which are built out of abelian varieties and reductive groups. 
For example, there are infinitely many semi-abelian
schemes over $\ZZ$, up to $\ZZ$-isomorphism. Explicitly, such schemes occur as the identity components of the  N\'{e}ron models 
of the elliptic curves $$y^2  +xy = x^3+n, \quad n\in \ZZ \setminus \{0\},$$  since these curves are semi-stable over $\ZZ$.
Moreover, the analogue of Theorem \ref{thm:reductive} is in fact
\emph{false} for abelian varieties (see \cite[p.~241]{Maz86}). 
Namely, there exists a non-empty finite set of places $S$ of $\QQ$, such that the set of $\QQ$-isomorphism classes of genus one curves over $\QQ$ with good reduction outside  $S$ is infinite (note that such curves are torsors for their Jacobians).
This is due to the fact that, for an elliptic curve $E$ over $\QQ$, the ``$S$-Tate-Shafarevich group''
$$\Sha_S(E) = \ker\left(H^1(\QQ,E) \to \prod_{p \not \in S} H^1(\QQ_p,E)\right)$$ 
can be infinite.
In particular, our results for reductive groups are actually \emph{stronger} 
than the corresponding results for abelian varieties.

\subsection*{Acknowledgements} The authors would like to thank  Jilong Tong for his help in Section \ref{sec:unipotent} and Ben Webster for his help in proving Lemma \ref{lem:rank}. The authors would also like to thank  Manfred Lehn, Adam Morgan, Nikita Semenov, Duco van Straten and Ronan Terpereau for helpful discussions, and the anonymous referee for numerous useful comments. The first named author gratefully acknowledges the support of SFB/Transregio 45.

\section{Reductive groups and torsors}
The aim of this section is to prove Theorem \ref{thm:reductive}.

\subsection{Finiteness results in \'{e}tale cohomology}\label{section: GMB}
In this paper we shall achieve our finiteness results by using a relationship
between good reduction and \'{e}tale cohomology. We gather here the finiteness results we shall require.
Some of these results hold in greater generality than stated, but for  simplicity we only work in the generality 
which we require.

Let $B$ be a dense open subscheme of $\Spec \OO_K$, where $K$ is a number field and $\OO_K$ is its ring of integers. Note that there exists a finite set of finite places $S$ of $K$ such that $B = \Spec \OO_K[S^{-1}]$.
If $G$ is a smooth separated 
group scheme over $B$, we denote by
$\check{\mathrm{H}}^1(B,G)$ the \v{C}ech cohomology set of $B$ with coefficients in $G$
with respect to the \'etale topology \cite[\S III.2]{MilneEC}.  Recall that a faithfully flat and locally of finite type $B$-scheme $E$ is a \emph{$G$-torsor} if it is endowed with a left action of $G$ such that the morphism \[E\times_B G \to E\times_{B} E, \quad (x,g)\mapsto (x,x\cdot g)\] is an isomorphism. By \cite[Thm.~III.4.3]{MilneEC},
\cite[Prop.~III.4.6]{MilneEC} and \cite[Thm.~XI.3.1]{Ray70}), the pointed set $\check{\mathrm H}^1(B,G)$ classifies $G$-torsors over $B$. Here
 we have the following finiteness theorem due to Gille and Moret-Bailly.
\begin{lemma}\label{thm: GMB onedim}
	Let $B \subset \Spec \OO_K$ be a dense open subcheme and let $G$ be a smooth affine group scheme of finite type over $B$. Then $\check{\mathrm{H}}^1(B,G)$ is finite.
\end{lemma}
\begin{proof}
	This is a special case of \cite[Prop.~5.1]{GMB13}.
\end{proof}

If $G_1$ and $G_2$ are group schemes over $B$, we shall say that $G_2$ is a \emph{twist} of $G_1$
if $G_2$ is $B$-isomorphic to $G_1$ locally for the \'etale topology on $B$.
If $\mathrm \Aut_B G_1$ is representable by a smooth separated group scheme,
then $\check{\mathrm H}^1(B,\Aut_B G_1)$ classifies the twists of $G_1$ (see \cite[\S III.4]{MilneEC}). 

 
We shall require an analogue of Lemma \ref{thm: GMB onedim} for some
constant group schemes which are not of finite type. 
For an (abstract) group $G$, we let $G_B$ denote the constant group scheme over $B$
associated to $G$. As a scheme this is a disjoint union $\coprod_{g \in G} B$,
and the group scheme structure is defined in the obvious way.
Note that $G_B$ is faithfully flat and locally of finite type over $B$,
but it is of finite type over $B$ if and only if 
$G$ is finite.
\begin{lemma}\label{lem: jz}
Let $B \subset \Spec \OO_K$ be a dense open subscheme and let $n$ be a positive integer. Then  $\check{\mathrm H}^1(B,\GL_n(\ZZ)_B)$ is finite.
\end{lemma}
\begin{proof}
As $\GL_n(\ZZ)$ is an arithmetic group,  it is finitely presented and  contains only finitely many finite subgroups, up to conjugacy; see \cite[Thm.~4.2]{PR94} and \cite[Thm.~4.3]{PR94}.
It is therefore ``decent'' in the sense of \cite[\S 16]{Maz93}, hence the result follows from \cite[Lem.~(a), \S 16]{Maz93} (this being a simple application
of Hermite-Minkowski).
\end{proof}

\subsection{Reductive algebraic groups}\label{sec:reductive}
We assume that the reader is familiar with the basic theory of reductive group schemes,
as found in  \cite{Conrad} or \cite{SGA3}.

\begin{definition}\label{def:reductive} Let $B$ be a scheme.
	A group scheme $G$ over $B$ is \emph{reductive} 
	if the geometric  fibres  of $G\to B$ are smooth connected affine with trivial unipotent radical.
\end{definition}

If $B$ is integral, we define the \emph{rank} of $G$ to be the dimension of the maximal torus of the generic fibre
(note that we do not consider only  split maximal tori).

\subsection{Tori and semi-simple groups}
In order to prove Theorem \ref{thm:reductive}, we first handle the cases of algebraic
tori and semi-simple groups.

\begin{proposition} \label{prop:tori}
	Let $n \in \NN$ and let $B \subset \Spec \OO_K$ be a dense open subscheme.
	Then the set of $B$-isomorphism classes of $n$-dimensional tori over $B$ is finite.
\end{proposition}
\begin{proof}
	Let $T$ be an $n$-dimensional torus over $B$. Then by definition, $T$
	is a twist of $\mathbb G_{m,B}^n$. Such twists are classified by $\check{\mathrm H}^1(B, \Aut_B \mathbb G_{m,B}^n)$.
	However $\Aut_B \mathbb G_{m,B}^n \cong \GL_n(\ZZ)_B$, hence the result follows from  Lemma \ref{lem: jz}.
\end{proof}

We next treat the case of semi-simple groups which are twists of a fixed \emph{split}
semi-simple group.

\begin{lemma}\label{lem: twisting lemma} Let $B \subset \Spec \OO_K$ be a dense open subscheme and
	let $G_K$ be a split semi-simple group over $K$. Then the set of $B$-isomorphism classes of semi-simple
	group schemes $G^\prime$ over $B$ such that $G^\prime_K$ is a twist of $G_K$ is finite.
\end{lemma}
\begin{proof}
	Note that as $G_K$ is split, by the existence theorem \cite[Cor.~XXV.1.3]{SGA3} there exists a smooth semi-simple group scheme $G$ over $B$ such that $G_K$ is isomorphic to $G\times_B K$.  Let $G'$ be a semi-simple group scheme over $B$ whose generic fibre
	is a twist of $G_K$ over $K$. 

	Locally for the \'etale topology on $B$, by \cite[Cor.~ XXII.2.3]{SGA3} we see that $G$ and $G^\prime$ are $B$-split and have isomorphic generic fibres over $K$. Therefore, by the isomorphism theorem \cite[Thm.~XX.4.1]{SGA3},  
	$G'$ is a twist of $G$, hence corresponds to a some element of $\check{\mathrm H}^1(B, \Aut_B G)$.
	However the scheme $\Aut_B G$ is smooth affine of finite type over $B$
	\cite[Thm.~XXIV.1.3]{SGA3}, hence 
	the result follows from Lemma \ref{thm: GMB onedim}.
\end{proof}


To continue, we need to know that there are only
finitely many \emph{split} semi-simple algebraic groups of fixed dimension over 
a given field. This follows from the classification of semi-simple groups,
as we now explain. We also
prove the analogous result for semi-simple algebraic groups of fixed rank,
as it will be used later on when we handle flag varieties.

\begin{lemma}\label{lem:ss_split}
	Let $n \in \NN$ and let $K$ be a field. 
	Then the set of $K$-isomorphism classes of split semi-simple group schemes over
	$K$ of dimension $n$ (or rank~$n$) is finite.
\end{lemma}
\begin{proof}
	\emph{The simply connected case.}
	Any split simply connected semi-simple algebraic group $G$ may be uniquely
	written as a finite direct product of \emph{simple} split simply connected 
	semi-simple algebraic groups over $K$. It therefore suffices
	to consider the case where $G$ is simple. 
	Such groups are classified by irreducible Dynkin diagrams. There
	are five exceptional diagrams, plus four infinite series $A_{\ell},B_{\ell},C_{\ell}$ and $D_{\ell}$,
	whose corresponding algebraic groups have dimensions 
	$\ell^2 +2\ell, \ell(2\ell+1) , \ell(2\ell+1)$ and $\ell(2\ell -1)$ 
	and ranks $\ell$, respectively 
	(these are the dimensions and ranks of $\SL_{\ell + 1}, \SO_{2\ell + 1},
	\SP_{2\ell}$ and $\SO_{2\ell}$, respectively).
	From this classification, it is clear that there are only finitely many
	simple split simply connected semi-simple algebraic groups	of fixed  dimension
	(or fixed rank).
	
	\emph{The general case.}
	Any semi-simple algebraic group $G$ fits into an exact sequence of algebraic groups
	$$0 \to \pi_1(G) \to \widetilde{G} \to G \to 0.$$
	Here $\widetilde{G}$ denotes the universal cover of $G$
	and $\pi_1(G)$ the algebraic fundamental group of $G$,
	which is finite and contained in the centre of $\widetilde G$.
	However the centre of a semi-simple algebraic group is finite,
	and so the result follows from the simply connected case.
\end{proof}

\begin{proposition}\label{prop:ss}
	Let $n \in \NN$ and let $B \subset \Spec \OO_K$ be a dense open subscheme.
	Then the set of $B$-isomorphism classes of semi-simple group schemes
	over $B$ of dimension $n$ (or rank $n$) is finite.
\end{proposition}
\begin{proof}
	This follows immediately from Lemma \ref{lem: twisting lemma}
	and Lemma \ref{lem:ss_split}.
\end{proof}

\subsection{Proof of Theorem \ref{thm:reductive}.} 
We now come to the proof of Theorem \ref{thm:reductive}.
We begin by handling torsors under a fixed reductive group scheme.

\begin{lemma}\label{lem: torsors}
	Let $B \subset \Spec \OO_K$ be a dense open subscheme and let $G$ be a reductive group scheme over $B$.
	Then the set of $B$-isomorphism classes of $G$-torsors over $B$ is finite. 
\end{lemma}
\begin{proof}
	As explained in \S \ref{section: GMB}, such torsors are classified by $\check{\mathrm H}^1(B,G)$.
	This set is finite by Lemma \ref{thm: GMB onedim}.
\end{proof}

Therefore to prove Theorem \ref{thm:reductive}, it suffices to show that the set of $B$-isomorphism classes of reductive group schemes
over $B$ of dimension $n$ is finite. To do so, let $G$ be a reductive group scheme over $B$ and let
$\mathscr{R}(G)$ and $\mathscr{D}(G)$ denote the radical and the derived subgroup scheme of $G$,
respectively (see \cite[Exp.~XXII]{SGA3}).
By \cite[Exp.~XXII, 6.2.3]{SGA3}, we have the following short exact sequence 
$$0 \to \mathscr{R}(G) \cap \mathscr{D}(G) \to \mathscr{R}(G) \times_B \mathscr{D}(G) \to G \to 0$$
of group schemes over $B$. Note that $\mathscr{R}(G) \cap \mathscr{D}(G)$ is
a central finite flat subgroup scheme of $G$.
As $\mathscr{R}(G)$ is a torus and $\mathscr{D}(G)$ is semi-simple, it follows
from Proposition \ref{prop:tori} and Proposition \ref{prop:ss} that, as
$G$ runs over all reductive group schemes over $B$ of fixed dimension, there are only
finitely many choices for $\mathscr{R}(G)$ and $\mathscr{D}(G)$.
Thus, to prove Theorem \ref{thm:reductive}, it suffices to show the following.

\begin{lemma}
	Let $T$ be an algebraic torus  and let $D$ be a semi-simple algebraic group
	scheme, both over $B$. Then there are only finitely many
	central finite flat subgroups $F \subset T \times_B D$ such that the
	projections
	$$F \to T, \quad F \to D,$$
	are closed immersions.
\end{lemma}
\begin{proof}
	As the centre of $D$ is a finite flat group scheme over $B$,
	the degree  of $F$ over $B$ is bounded.
	The result therefore follows from the fact that $T$ contains only finitely many 
	finite flat subgroup schemes of
	bounded order, since $T[n]$ is finite flat for each $n$.  \end{proof}

 This completes the proof of Theorem \ref{thm:reductive}. \qed


\section{Flag varieties} \label{sec:flags}
The aim of this section is to prove Theorem \ref{thm:flags}. 

\subsection{Preliminaries}
 We begin with some preliminaries on flag varieties.

\begin{definition} \label{def:flag}
	 A smooth proper variety over a field $K$ is a \textit{(twisted) flag variety over $K$} 
	 if it is a homogeneous space for some reductive algebraic group scheme over $K$, and the stabiliser
	 of each point is smooth.
	 A smooth proper scheme over a scheme  $B$ is a \textit{flag scheme over $B$} if its 
	geometric fibres are flag varieties.
\end{definition}

Note that a flag variety may be a homogeneous space for many reductive groups, e.g.~the projective line $\PP^1$
is a homogeneous space for both $\SL_2$ and $\PGL_2$.
Demazure proved that flag schemes are in fact homogeneous spaces for their automorphism group schemes.  

\begin{lemma}\label{lem: dem} 
	Let $X$ be a flag scheme over a scheme $B$. Then $\Aut_B(X)$ is representable
	by a smooth affine group scheme whose neutral component $\Aut_B (X)^0$ is semi-simple.
	Moreover $X$ is a homogeneous space for $\Aut_B (X)^0$.
\end{lemma}
\begin{proof}
	This is  an application of \cite[Prop.~6.4]{Dem77}.
\end{proof}

Let $G$ be a  semi-simple group 
scheme over a scheme $B$. Recall that a smooth $B$-subgroup scheme $P \subset G$ is called \emph{parabolic} 
if the quotient $G / P$ is a proper $B$-scheme.
In this case $G/P$ is  obviously a flag scheme. Over an algebraically closed
field every flag variety is of this form (one simply takes $P$ to be
the stabiliser of a rational point), but this does not have to be the
case over more general bases.

\begin{lemma} \label{lem:finitely_many_flags}
	Let $G$ be a semi-simple group scheme over a Noetherian scheme $B$.
	Then the set of $B$-isomorphism classes of flag
	schemes which are a homogeneous space for $G$ is finite.
\end{lemma}
\begin{proof}
 
Define the functor $\underline{\Par}_G$ from the category of schemes over $B$ to the category of sets by
$$\underline{\Par}_{G}: S \mapsto \{ \textrm{Parabolic subgroups of } G_S\}.$$
This functor is representable by a smooth projective scheme over $B$, which we denote by $\Par_G$
(see \cite[Cor.~5.29]{Conrad} or \cite[Cor.~XXVI.3.5]{SGA3}). 
 Note that $G$ acts on $\Par_G$ by conjugation.  In particular, the irreducible components
of $\Par_G$ are flag schemes for $G$. 	Now let $X$ be a flag scheme for $G$. We  define a morphism from $X$
	to $\Par_G$ by
	\begin{align*}
		X(S) &\to \Par_G(S) \\
		x &\mapsto \textrm{stabiliser of } x,
	\end{align*} for any $B$-scheme $S$. As this morphism is $G$-equivariant, it identifies 
	$X$ with an irreducible component of $\Par_G$.  As $\mathrm{Par}_G$ is Noetherian, the result is proved.
\end{proof}

One also has control over the dimension of a flag variety in terms of the rank of
its automorphism group.

\begin{lemma} \label{lem:rank}
	Let $X$ be a flag variety over a field $K$.
	Then $$\dim X \geq \rank \Aut(X)^0.$$
\end{lemma}
\begin{proof}
	We may assume that $K$ is algebraically closed.
	Take $G=\Aut(X)^0$ and let $T \subset G$ be a maximal torus.
	The action of $G$ on $X$ is faithful,
	hence $T$ also acts faithfully on $X$.

	Borel's fixed point theorem \cite[Thm.~II.10.4]{Borel}
	implies that $T$ acts on $X$ with a fixed point $x \in X(K)$.
	We claim that the action of $T$ on the tangent
	space $\mathrm{Tan}_{x}(X)$ of $x$ is faithful. 
	Indeed, let $H = \ker (T \to \GL(\mathrm{Tan}_{x}(X)))$.
	By \cite[Prop.~A.8.10]{CGP}, the fixed locus $X^H$
	is a smooth closed subscheme of $X$ and $\mathrm{Tan}_{x}(X)^H = \mathrm{Tan}_{x}(X^H)$.
	As $H$ acts trivially on $\mathrm{Tan}_{x}(X)$ we find that $\dim X^H = \dim X$, hence $X^H = X$.
	However the action of $T$ on $X$ is faithful, thus $H$ is trivial.
	This proves the claim.
	
	On diagonalising the action of 
	$T$ on $\mathrm{Tan}_{x}(X)$, we therefore obtain
	\[\dim X = \dim  \mathrm{Tan}_{x}(X) \geq \dim T = \rank G, \]
	as required.
	\end{proof}
For flag varieties, we use the following notion of good reduction.

\begin{definition} \label{def:good_reduction_flag} 
	Let $K$ be a number field and let $S$ be a finite set of finite places of $K$.
	A flag variety $X$  over $K$  has \textit{good reduction outside  $S$} if there exists a flag
	scheme $\mathcal{X}$ over $\OO_K[S^{-1}]$ whose generic fibre is isomorphic to~$X$.
\end{definition}

\subsection{Proof of Theorem \ref{thm:flags}} 

In order to prove Theorem \ref{thm:flags}, it suffices to show the following.

\begin{theorem} \label{thm:flags_sch}
	Let $K$ be a number field, let $n\in \NN$ and let $B \subset \Spec \OO_K$ be a dense
	open subscheme. Then the set of $B$-isomorphism classes of $n$-dimensional
	flag schemes over $B$ is finite.
\end{theorem}
\begin{proof}
	Let $X$ be a flag scheme over $B$ of dimension $n$.
	By Lemma \ref{lem: dem}, we known that $\Aut_B (X)^0$ is a semi-simple group scheme over $B$
	and that  $X$ is a flag scheme for $\mathrm{Aut}_B(X)^0$. 
	Moreover by Lemma \ref{lem:rank}, we have 
	$$\rank \mathrm{Aut}_B(X)^0 \leq n.$$
	Hence by Proposition \ref{prop:ss} there are only finitely many choices for $\mathrm{Aut}_B(X)^0$.
	Moreover, by Lemma \ref{lem:finitely_many_flags}, there are only finitely many flag schemes over $B$
	associated to $\mathrm{Aut}_B(X)^0$. This proves the result.
\end{proof}
This completes the proof of Theorem \ref{thm:flags}. \qed

\section{Unipotent groups} \label{sec:unipotent} 
We now explain how the analogues of our main results do not hold for unipotent
groups.

\subsection{Unipotent group schemes} \label{sec:unipotent_sch}
We first show that the analogue of Corollary \ref{cor:reductive_sch} fails for unipotent groups of dimension at least two. 
Let $p$ be a prime and let $W_p$ be the additive group of Witt vectors
of length $2$ over $\ZZ$. Explicitly, this is a smooth commutative
unipotent scheme over $\ZZ$ with group law
\begin{align*}
	W_p \times W_p &\to W_p \\
	(x,y) \times (x',y') & \mapsto \left(x + x', y + y' + \frac{(x + x')^p - x^p - x'^p}{p}\right).
\end{align*}
For a ring $R$, we have $W_{p}\times R \cong \mathbb G_{a,R}^2$ if and only if $p$ is invertible in $R$.
Hence if $p$ and $q$ are distinct prime numbers, then $W_{p}$ is not isomorphic to $W_{q}$. 
We therefore see that there are infinitely many smooth commutative unipotent group schemes of
dimension $2$ over $\ZZ$, up to isomorphism. 

For completeness let us note that the analogue of Corollary \ref{cor:reductive_sch}  does indeed
hold in the one-dimensional case. Namely, let $B$ be a dense open subscheme of $\Spec \OO_K$.
Then by \cite[Cor.~3.8]{WW80}, the set of $B$-isomorphism classes of one-dimensional unipotent groups 
over $B$ is in bijection with  $\Pic(B)$. In particular, this set is finite.  

\subsection{Good reduction of unipotent groups}
We now show that the analogue
of Corollary \ref{cor:reductive_gr} fails for unipotent groups.
Here we say that a unipotent group $U$ over $K$ has good reduction outside
a finite set of finite places $S$ of $K$, if there exists a smooth
unipotent group scheme over $\OO_K[S^{-1}]$ whose generic fibre is isomorphic
to $U$.
We construct counter-examples by exploiting a relationship between nilpotent
Lie algebras and unipotent algebraic groups. For each $\lambda \in \ZZ$, 
we   define a $7$-dimensional 
nilpotent Lie algebra $\mathfrak{u}_\lambda$ over $\ZZ$ as follows. As a $\ZZ$-module $\mathfrak u_\lambda$ has a 
basis $x_1,\ldots,x_7$, and the Lie bracket is given by
\begin{align*}
	[ x_1,x_i] & = x_{i+1} \quad (2 \leq   i \leq 6), \qquad
	[ x_2,x_3]  = x_5,  \qquad [ x_2,x_4]  = x_6,\\
	[ x_2,x_5] & = \lambda x_7, \qquad \qquad \qquad \qquad  [ x_3,x_4]  = (1 - \lambda)x_7. 
\end{align*}
These are the Lie algebras with label $123457_I$ in \cite{See93}.
We chose these Lie algebras as their generic fibres are 
pairwise non-isomorphic over $\bar \QQ$.

A standard argument using the Baker-Campbell-Hausdorff formula 
(see \cite[\S IV.2.4]{DG70}) yields a smooth unipotent  group scheme
$U_\lambda$ over $\ZZ[1/7!]$ whose Lie algebra is $\mathfrak{u}_\lambda \otimes \ZZ[1/7!]$.   In particular, we obtain infinitely many pairwise
non-$\Qbar$-isomorphic unipotent algebraic groups over $\QQ$ of dimension $7$ which have good reduction
outside  $\{2,3,5,7\}$.


\begin{thebibliography}{10}

\bibitem{Borel}
A.~Borel.
\newblock {\em Linear algebraic groups}, volume 126 of {\em Graduate Texts in
  Mathematics}.
\newblock Springer-Verlag, New York, second edition, 1991.

\bibitem{BS64}
A.~Borel and J.-P. Serre.
\newblock Th\'eor\`emes de finitude en cohomologie galoisienne.
\newblock {\em Comment. Math. Helv.}, 39:111--164, 1964.

\bibitem{Conrad}
B.~Conrad.
\newblock Reductive group schemes.
\newblock {\em Available on author's homepage}.

\bibitem{CGP}
B.~Conrad, O.~Gabber, and G.~Prasad.
\newblock {\em Pseudo-reductive groups}, volume~17 of {\em New Mathematical
  Monographs}.
\newblock Cambridge University Press, Cambridge, 2010.

\bibitem{Dem77}
M.~Demazure.
\newblock Automorphismes et d\'eformations des vari\'et\'es de {B}orel.
\newblock {\em Invent. Math.}, 39(2):179--186, 1977.

\bibitem{DG70}
M.~Demazure and P.~Gabriel.
\newblock {\em Groupes alg\'ebriques. {T}ome {I}}.
\newblock North-Holland Publishing Co., Amsterdam, 1970.

\bibitem{SGA3}
M.~Demazure and A.~Grothendieck.
\newblock {\em Sch\'emas en groupes I, II, III ({SGA} 3).}
\newblock Lecture Notes in Math. 151, 152, 153. Springer-Verlag, New York,
  1970.

\bibitem{Fal83}
G.~Faltings.
\newblock Endlichkeitss\"atze f\"ur abelsche {V}ariet\"aten \"uber
  {Z}ahlk\"orpern.
\newblock {\em Invent. Math.}, 73(3):349--366, 1983.

\bibitem{GMB13}
P.~Gille and L.~Moret-Bailly.
\newblock Actions alg\'ebriques de groupes arithm\'etiques.
\newblock In {\em Torsors, \'etale homotopy and applications to rational
  points}, volume 405 of {\em London Math. Soc. Lecture Note Ser.}, pages
  231--249. Cambridge Univ. Press, 2013.

\bibitem{JL}
A.~Javanpeykar and D.~Loughran.
\newblock Good reduction of complete intersections.
\newblock {\em Preprint}.

\bibitem{Maz86}
B.~Mazur.
\newblock Arithmetic on curves.
\newblock {\em Bull. Amer. Math. Soc.}, 14(2):207--259, 1986.

\bibitem{Maz93}
B.~Mazur.
\newblock On the passage from local to global in number theory.
\newblock {\em Bull. Amer. Math. Soc. (N.S.)}, 29(1):14--50, 1993.

\bibitem{MilneEC}
J.~S. Milne.
\newblock {\em \'{E}tale cohomology}, volume~33 of {\em Princeton Mathematical
  Series}.
\newblock Princeton University Press, Princeton, N.J., 1980.

\bibitem{PR94}
V.~Platonov and A.~Rapinchuk.
\newblock {\em Algebraic groups and number theory}, volume 139 of {\em Pure and
  Applied Mathematics}.
\newblock Academic Press, Inc., Boston, MA, 1994.

\bibitem{Ray70}
M.~Raynaud.
\newblock {\em Faisceaux amples sur les sch\'emas en groupes et les espaces
  homog\`enes}.
\newblock Lecture Notes in Mathematics, Vol. 119. Springer-Verlag, Berlin-New
  York, 1970.

\bibitem{Shaf1962}
I.~R. {\v{S}}afarevi{\v{c}}.
\newblock Algebraic number fields.
\newblock In {\em Proc. {I}nternat. {C}ongr. {M}athematicians ({S}tockholm,
  1962)}, pages 163--176. Inst. Mittag-Leffler, Djursholm, 1963.

\bibitem{Scholl}
A.~J. Scholl.
\newblock A finiteness theorem for del {P}ezzo surfaces over algebraic number
  fields.
\newblock {\em J. London Math. Soc. (2)}, 32(1):31--40, 1985.

\bibitem{See93}
C.~Seeley.
\newblock {$7$}-dimensional nilpotent {L}ie algebras.
\newblock {\em Trans. Amer. Math. Soc.}, 335(2):479--496, 1993.

\bibitem{Sti09}
H.~Stichtenoth.
\newblock {\em Algebraic function fields and codes}, volume 254 of {\em
  Graduate Texts in Mathematics}.
\newblock Springer-Verlag, Berlin, second edition, 2009.

\bibitem{WW80}
W.~C. Waterhouse and B.~Weisfeiler.
\newblock One-dimensional affine group schemes.
\newblock {\em J. Algebra}, 66(2):550--568, 1980.

\end{thebibliography}
\end{document}